\documentclass[10pt, leqno]{amsart}
\usepackage{array}
\usepackage{longtable}
\usepackage{url}
\usepackage{enumerate}
\usepackage{amsmath}
\usepackage{amssymb}
\usepackage{amsthm}
\usepackage{amsfonts}
\usepackage{tikz,pgf}
\usepackage{graphicx}
\usepackage{caption}
\usepackage{subcaption}
\usepackage{ulem}

\allowdisplaybreaks

  \usepackage{ifthen}
 \newcommand{\mymarginpar}[1]{%
    \marginpar{\ifthenelse{\isodd{\arabic{page}}}{\flushleft 
#1}{\flushright #1}}}

\numberwithin{equation}{section}

\newcommand{\sing}[1]{\text{sing}(#1)}
\newcommand{\reg}[1]{\text{reg}(#1)}
\newcommand{\source}[1]{\text{source}(#1)}
 
\newcommand{\IN}{\mathbb{N}}

\newcommand{\IT}{\mathbb{T}}                  
\newcommand{\IZ}{\mathbb{Z}}

\newcommand{\CK}{\mathcal{K}}
\newcommand{\CT}{\mathcal{T}}

 \theoremstyle{plain} 
 \newtheorem{Theorem}{Theorem}[section]
 \newtheorem{Lemma}[Theorem]{Lemma}
 \newtheorem{Proposition}[Theorem]{Proposition}
 
 \newtheorem{Corollary}[Theorem]{Corollary}
 
 \theoremstyle{definition} 
 \newtheorem{Definition}[Theorem]{Definition}
 \newtheorem{Remark}[Theorem]{Remark}
 
 \newtheorem{Example}[Theorem]{Example}

\newtheorem{Notation}[Theorem]{Notation}

\begin{document}
\author{Samantha Brooker}
\author{Jack Spielberg}
\title{Relative graphs and pullbacks of relative Toeplitz graph algebras}
\address{School of Mathematical and Statistical Sciences \\ Arizona State University \\ P.O. Box 871804 \\ Tempe, AZ 85287-1804}
\email{sjbrooke@asu.edu}
\email{jack.spielberg@asu.edu}

\begin{abstract}

In this note we generalize a result from a recent paper of Hajac, Reznikoff and Tobolski (\cite{hrt}). In that paper they give conditions they call \textit{admissibility} on a pushout diagram in the category of directed graphs implying that the $C^*$-algebras of the graphs form a pullback diagram. We consider a larger category of \textit{relative graphs} that correspond to relative Toeplitz graph algebras. In this setting we give necessary and sufficient conditions on the pushout to get a pullback of $C^*$-algebras.

\end{abstract}

\maketitle

\section{Introduction}
\label{sec intro}

In a recent paper of Hajac, Reznikoff and Tobolski (\cite{hrt}) the authors give conditions they call \textit{admissibility} on a pair of subgraphs of a directed graph implying that the $C^*$-algebras of the three graphs fit into a pullback diagram. The admissibility implies, in particular, that the $C^*$-algebra of each subgraph is a quotient of the $C^*$-algebra of the ambient graph, or, equivalently, that the complement of the vertex set of each subgraph is \textit{saturated} and \textit{hereditary}. In order to accomodate graphs that are not row finite they also require that the hereditary sets be \textit{unbroken}, meaning that neither hereditary set admits any \textit{breaking vertices} (see Definition \ref{def saturated}). The breaking vertices form a crucial part of the classification of gauge-invariant ideals of graph algebras (\cite{bhrs}). Admissibility in \cite{hrt} implies that the saturated hereditary set corresponding to each subgraph is associated to a unique gauge-invariant ideal in the $C^*$-algebra of the ambient graph, and that the quotient by this ideal is isomorphic to the $C^*$-algebra of the subgraph. It is shown in \cite{bhrs} that in the presence of breaking vertices, a saturated hereditary set is associated with several gauge-invariant ideals, corresponding to the subsets of the breaking vertices. Except in the case where the entire set of breaking vertices is used, the quotient $C^*$-algebra will not equal the algebra of a subgraph. In \cite{bhrs}, for each gauge-invariant ideal a new graph is constructed whose $C^*$-algebra does equal the quotient by that ideal. However, these quotients can in fact be presented as $C^*$-algebras associated to actual subgraphs by considering \textit{relative Toeplitz graph algebras} (see Definition \ref{def relative toeplitz algebra}). (Relative Toeplitz graph algebras have been considered previously in \cite{spi1},\cite{mt},\cite{sims},\cite{ek}.)

In this paper we show that the question of obtaining a pullback diagram of $C^*$-algebras from a pushout diagram of graphs extends naturally to the context of relative Toeplitz graph algebras. This makes use of all of the gauge-invariant ideals of $\CT C^*(E)$. We introduce a corresponding category of \textit{relative directed graphs} and a contravariant functor from it to $C^*$-algebras. This is an extension of the results of \cite{bhrs}. We prove that pushouts exist in this category, and then characterize those pushout diagrams that give rise to pullback diagrams of $C^*$-algebras. We call such diagrams \textit{admissible} in imitation of the notion in \cite{hrt}, although it is quite different from that definition. Essentially, most of the properties coming from the admissibility conditions of \cite{hrt} are incorporated in our definition of morphism in the category of relative directed graphs. Our admissibility condition has to do with the variety of gauge-invariant ideals available in the case of arbitrary graphs.

We briefly describe the contents of the paper. In section \ref{section relative graph algebras} we recall the terminology of directed graphs as used in operator algebras. In particular we recall the relative Toeplitz graph algebras. The standard definition is by means of generators and relations (the \textit{Cuntz-Krieger relations}). However these algebras are also described as $C^*$-algebras of certain \'etale groupoids, and we find it advantageous to work with the groupoid description.  In section \ref{section groupoid picture} we recall this description as given in \cite{spi2}. (This is slightly different from the original groupoid description.) We give the definitions of the usual graph algebras and Toeplitz graph algebras, as well as that of the relative Toeplitz graph algebras, and describe the ideals relevant to this paper in terms of open invariant subsets of the unit space of the groupoid. In section \ref{section gauge-invariant ideals} we prove that these are, in fact, all of the gauge-invariant ideals in the relative Toeplitz graph algebras, giving a natural extension of the results of \cite{bhrs}. We have tried to keep this argument as general as possible for as long as possible, so that the specifics of graph algebras are needed only in the last part of the argument. In the end we give a simple characterization of the gauge-invariant ideals in terms of subsets of the graph, generalizing the characterization of such ideals in graph algebras from \cite{bhrs}. We note that while a more general theorem is proved in \cite{sww} for higher rank graphs, the characterization in that generality is quite complicated, and indeed, even translating it to the case of directed graphs is complicated. Thus we feel that our direct proof and simple characterization for directed graphs justify the inclusion of our treatment. An equivalent version of this classification was given in \cite[Theorem 2.11]{ek}, with a simple proof using the result of \cite{bhrs}. (We thank the anonymous referee for bringing the article \cite{ek} to our attention.) Our treatment bypasses the result of \cite{bhrs}, and we hope that our methods may be useful in more general situations. In section \ref{section relative graphs} we introduce a new category of \textit{relative graphs}.  The objects are pairs consisting of a directed graph and a certain subset of the vertex set. The definition is exactly what is needed to describe the relative Toeplitz graph algebras in terms of the graphs. We prove that it is indeed a category, and that pushouts exist in this category. In section \ref{section admissible pushouts of relative graphs} we first recall a theorem of Pedersen to characterize pullback diagrams with surjective maps. We then give our definition of admissible pushout diagram in the category of relative graphs, and use it to prove our main result, Theorem \ref{thm main}. In the final section \ref{section examples} we consider several examples to illustrate the main theorem. In particular we consider the two extreme situations: where all four algebras are Toeplitz graph algebras, and where all four algebras are graph algebras. In the second of these we compare our theorem with the theorem of \cite{hrt}.

We wish to place our results in the context of the paper \cite{kpsw}. That paper treats pullbacks of the Toeplitz and Cuntz-Krieger algebras of finitely aligned higher-rank graphs (without sources) by means of a connected sum of the graphs. Apart from the restriction to sourceless graphs (which is a reasonable restriction in the higher-rank case), this is much more general than the situation of this paper. The new feature in our treatment is to consider pullbacks of relative Toeplitz graph algebras. It would be interesting to work out the pullback structure for relative Toeplitz algebras of higher-rank graphs using the results of \cite{sww}.

Throughout this paper we use the \textit{Australian convention} for the $C^*$-algebras of directed graphs, as presented in \cite{rae}. Thus edges of the graph are thought of as morphisms in a (small) category, and concatenation of edges corresponds to composition of morphisms. The result is that when the vertices and edges of a directed graph are represented as projections and partial isometries in a $C^*$-algebra, the source vertex of an edge is represented by the projection corresponding to the initial space of the partial isometry representing that edge.

\section{Relative Toeplitz graph algebras} \label{section relative graph algebras}

\begin{Definition}
A \textit{directed graph} is a quadruple $E = (E^0, E^1, r, s)$, where $E^0$ is the set of vertices, $E^1$ is the set of directed edges, and $r, s: E^1 \to E^0$ are the range and source maps, respectively, so that if $e \in E^1$ is an edge from $v$ to $w$, then $r(e) = w$ and $s(e) = v$. We then say that $v$ \textit{emits} the edge $e$, and that $w$ \textit{receives} the edge $e$. We let $E^n$ denote the set of directed paths in $E$ consisting of $n$ edges, $E^* = \bigcup_{n=0}^\infty E^n$, the set of all directed paths of finite nonnegative length, and $E^\infty$ the set of all (semi) infinite directed paths. Thus $\alpha \in E^n$ means that $\alpha = e_1 e_2 \cdots e_n$, where $e_1, \ldots, e_n \in E^1$ and $s(e_i) = r(e_{i+1})$ for $1 \le i < n$, and $x \in E^\infty$ means that $x = e_1 e_2 \cdots$, where for each $n$, $e_1 \cdots e_n \in E^n$. We extend the source and range maps to $E^*$ and $E^\infty$ by setting (in the above) $s(\alpha) = s(e_n)$, $r(\alpha) = r(x) = r(e_1)$. For $\alpha \in E^*$ we write $\alpha E^* := \{\alpha \beta : \beta \in E^*, r(\beta) = s(\alpha) \}$, and similarly for $E^* \alpha$, $\alpha E^\infty$, and if $x \in E^\infty$, for $E^* x$. Thus we may write $s^{-1}(v) = v E^1$ for $v \in E^0$ to avoid specifying to which version of $s$ we refer, and for $A \subseteq E^* \cup E^\infty$ we write $E^* A = \bigcup_{\alpha \in A} E^* \alpha$, etc. We write $|\alpha|$ for the \textit{length} of $\alpha$: $|\alpha| = n$ for $\alpha \in E^n$.

A \textit{source} is a vertex $v \in E^0$ that receives no edges, that is, $v E^1 = \emptyset$.

An \textit{infinite receiver} is a vertex $v \in E^0$ such that $v E^1$ is infinite. $E$ is called \textit{row-finite} if it has no infinite receivers.

A vertex $v$ is said to be \textit{singular} in $E$ if $v$ is either a source or an infinite receiver, and is said to be \textit{regular} in $E$ if it is not singular. We denote the set of singular vertices in $E$ by $\sing{E}$, the set of regular vertices in $E$ by $\reg{E}$, and the set of sources in $E$ by $\source{E}$.
\end{Definition}

All graphs we consider are directed graphs, so we will usually say ``graph'' rather than ``directed graph.''

\begin{Definition} \label{def relative toeplitz algebra}
Let $E$ be a graph and let $A$ be a $C^*$-algebra. A \textit{Cuntz-Krieger E-family} in $A$ consists of a set of projections $\{P_v: v \in E^0\} \subseteq A$ such that $P_v P_w = 0$ whenever $v \neq w$ (we say that $\{P_v\}$ is a set of \textit{mutually orthogonal projections}) and a set of partial isometries $\{S_e : e \in E^1\} \subseteq A$ satisfying the \textit{Cuntz-Krieger relations:}
\begin{itemize}
    \item[(CK1)] for all $e \in E^1$, $S_e^* S_e = P_{s(e)}$
    \item[(CK2)] for all $e$, $f \in E^1$, if $e \not= f$ then $S_e^*S_f = 0$
    \item[(CK3)] for all $e \in E^1$, $P_{r(e)} S_e = S_e$
    \item[(CK4)] for all $v \in \reg{E}$,
    $\displaystyle P_v =  \sum_{e \in v E^1} S_eS_e^*$.
\end{itemize}
The \textit{Cuntz-Krieger algebra}, or \textit{graph algebra}, of $E$ is the $C^*$-algebra generated by a universal Cuntz-Krieger $E$-family, and is denoted $C^*(E)$. Because of its special significance, we refer to the fourth relation (CK4) as \textit{the Cuntz-Krieger condition (at $v$)}. The \textit{Toeplitz graph algebra} of $E$ is the $C^*$-algebra defined as above but without the relation (CK4), and is denoted $\CT C^*(E)$. Finally, for a subset $A \subseteq \reg{E}$, the \textit{relative Toeplitz graph algebra} is defined as above but with the relation (CK4) applied only at vertices $v \in A$, and is denoted $\CT C^*(E,A)$. Thus $\CT C^*(E,A)$ is defined by the relations (CK1) - (CK3) together with 
\begin{itemize}
    \item[(TCK4)] for all $v \in A$, $\displaystyle P_v = \sum_{e \in v E^1} S_e S_e^*$.
\end{itemize}
\end{Definition}

We note that $\CT C^*(E) = \CT C^*(E,\varnothing)$, and  $C^*(E) = \CT C^*(E,\reg{E})$.

Central to the ideas of this paper is a contravariant functor from graphs to $C^*$-algebras, specifically relative Toeplitz graph algebras. These algebras were introduced in \cite{spi1} to describe subalgebras corresponding to subgraphs (see also \cite{mt}, \cite{ek}). From this, in fact, there is also a covariant functor from \textit{relative graphs} (see Definition \ref{definition relative graph}) to $C^*$-algebras (discussed in the purely algebraic context in \cite[Section 1.5]{aas}) but the present work does not require this.

To describe the contravariant functor that is of interest to us, we will need some more definitions.

\begin{Definition} \label{def hereditary}

A set $H \subseteq E^0$ is \textit{hereditary} if whenever $v \in H$ and $w \in s(v E^*)$ then $w \in H$. Let $E$ be a graph and let $H \subseteq E^0$ be a hereditary subset. Define a subgraph $F \equiv F_H$ of $E$ by letting $F^0 = E^0 \setminus H$ and $F^1 = F^0 E^1 F^0$ $(= E^1 F^0$ since $F^0$ is \textit{cohereditary}). For $v \in \reg{F}$ define $p_{v,H} = \sum_{e \in v F^1} s_e s_e^*$.

\end{Definition}

\begin{Definition} \label{def saturated}

$H$ is \textit{saturated} if for each vertex $v \in \reg{E}$, if $v E^1 = v E^1 H$ then $v \in H$, that is, if every vertex that sends an edge to $v$ is in $H$, then $v$ must be in $H$.

If $H$ is both hereditary and saturated and $F = F_H$, define $B_H = \reg{F} \cap \sing{E}$, the \textit{breaking vertices} for $H$. Thus $v$ is a breaking vertex for $H$ if $v$ receives infinitely many edges with source in $H$, and receives a finite nonzero number of edges with source in $F^0$ (note that $v$ must be a vertex of $F$ by the hereditary property of $H$). If necessary we may write $B^E_H$ to indicate the ambient graph. (In \cite{bhrs} the set $B_H$ is written $H_\infty^{\text{fin}}$.)

\end{Definition}

\begin{Remark} \label{remark saturated}

The definition of saturated $H \subseteq E^0$ can be stated equivalently as $\reg{E} \cap F_H^0 \subseteq \reg{F_H}$ (or $\reg{E} \cap \source{F_H} = \varnothing$).
The definition of hereditary $H$ can be stated equivalently as $H E^* = H E^* H$, or as $H E^1 = H E^1 H$.
\end{Remark}

We can now begin to describe the contravariant functor mentioned above. Let $E$ be a graph, let $H$ be a saturated hereditary subset of $E^0$, and let $F = F_H$. Let $J_H$ be the ideal in $C^*(E)$ generated by $\{p_v : v \in H\} \cup \{p_v - p_{v,H} : v \in B_H \}$. Then $C^*(E) / J_H \cong C^*(F)$. Thus the inclusion $F \hookrightarrow E$ induces a surjection $C^*(E) \twoheadrightarrow C^*(F)$. This is a special case of \cite[Theorem 3.6]{bhrs}, where all gauge-invariant ideals in $C^*(E)$ are described by means of the hereditary saturated subsets of $E^0$ and the subsets of the breaking vertices for these sets. We will generalize this below (Corollary \ref{cor kernel of quotient map}) using relative Toeplitz graph algebras. In this general setting, the ideals involved require just a hereditary subset of $E^0$. It is only when discussing ideals in Cuntz-Krieger algebras that the notion of saturation is relevant.

\section{The groupoid picture} \label{section groupoid picture}

We will describe the ideals that are the kernels of the quotient maps between relative Toeplitz algebras of a fixed graph. For this it is convenient to use the description of these algebras using groupoids. Recall that a groupoid is called \textit{ample} if it is \'etale and its unit space is totally disconnected.

For a directed graph $E$ we let $G(E)$ be the groupoid of $E$. We will use the description from \cite{spi2}. The unit space is $G(E)^{(0)} = E^\infty \cup E^*$, the set of all infinite and finite directed paths in $E$. For $\alpha \in E^*$ we let $Z(\alpha) = \alpha E^\infty \cup \alpha E^*$ be the set of all infinite and finite paths that extend $\alpha$. Then a base of compact-open sets for the topology of $G(E)^{(0)}$ is given by all sets of the form $Z(\alpha) \setminus \bigcup_{i=1}^n Z(\beta_i)$, for $\alpha$, $\beta_i \in E^*$. This topology is totally disconnected locally compact Hausdorff. The \textit{boundary} of $E$ is the subset $\partial E = E^\infty \cup E^* \sing{E}$, a closed subset. The complement $G(E)^{(0)} \setminus \partial E = E^* \reg{E}$ is a countable discrete open subset. We let $E^* * E^* * G(E)^{(0)} = \{(\alpha,\beta,x) \in E^* \times E^* \times G(E)^{(0)} : s(\alpha) = s(\beta) = r(x) \}$. 

The groupoid of $E$ is then $G(E) = E^* * E^* * G(E)^{(0)} \bigr/ \sim$, where $(\alpha,\beta,x) \sim (\alpha',\beta',x')$ if there are $\gamma, \gamma' \in E^*$ and $y \in G(E)^{(0)}$ such that $x = \gamma y$, $x' = \gamma' y$, $\alpha \gamma = \alpha' \gamma'$, and $\beta \gamma = \beta' \gamma'$. Inversion is given by $[\alpha,\beta,x]^{-1} = [\beta,\alpha,x]$. Composition is given as follows. It can (easily) be shown that if $\beta x = \alpha' x'$ then there are $y$, $\gamma$ and $\gamma'$ such that $x = \gamma y$, $x' = \gamma' y$, and $\beta \gamma = \alpha' \gamma'$. Then $[\alpha,\beta,x] \cdot [\alpha',\beta',x'] = [\alpha\gamma,\beta'\gamma',y]$ (see \cite{spi2}.) The range and source maps are given by $r([\alpha,\beta,x]) = \alpha x$ and and $s([\alpha,\beta,x]) = \beta x$ (or by $[r(\alpha),r(\alpha),\alpha x]$  and $[r(\beta),r(\beta),\beta x]$, if we identify $x \in G(E)^{(0)}$ with $[r(x),r(x),x] \in G(E)$). 

For $(\alpha,\beta) \in E^* * E^*$, and $F \subseteq Z(s(\alpha))$, we write $[\alpha,\beta,F] = \{[\alpha,\beta,x] : x \in F \}$. For $F$ compact-open in $G(E)^{(0)}$, these sets are compact-open bisections that form a base for a topology on $G(E)$ making it an ample Hausdorff groupoid. (The connection with the other formulation of the groupoid is given by using the triple $(\alpha x, |\alpha| - |\beta|, \beta x)$ in place of our triple $[\alpha,\beta,x]$.)

It is proved in \cite{spi2} that $C^*(G(E)) = \CT C^*(E)$, where $p_v = \chi_{Z(v)}$ (or more precisely, $\chi_{[v,v,Z(v)]}$) and $s_e = \chi_{[e,s(e),Z(s(e))]}$. Moreover, $\partial E$ is a closed invariant subset of $G(E)^{(0)}$, and $C^*(G(E)|_{\partial E}) = C^*(E)$, the usual graph algebra. There is a short exact sequence
\[
0 \to C^*(G(E)|_{E^*\reg{E}}) \to \CT C^*(E) \to C^*(E) \to 0
\]
(\cite[Remark 4.10]{ren}). For each $v \in \reg{E}$, the set $E^* v$ is open, discrete and invariant, and $G(E)|_{E^* v}$ is a principal transitive groupoid. Thus the ideal $C^*(G(E)|_{E^* \reg{E}})$ is isomorphic to $\bigoplus_{v \in \reg{E}} \CK(\ell^2(E^*v))$. The summand corresponding to $v \in \reg{E}$ is generated as an ideal by $\chi_{[v,v,\{v\}]} = p_v - \sum_{e \in v E^1} s_e s_e^*$, the \textit{gap projection} at $v$. (More explicitly, a typical element of $G|_{E^*v}$ is $[\alpha,\beta,\gamma v] = [\alpha\gamma,\beta\gamma,v]$, so $G|_{E^* v} = \{ [\xi,\eta,v] : \xi,\eta \in E^*v \} \equiv E^* v \times E^*v$.)

Let $A \subseteq \reg{E}$. Define $G(E,A)^{(0)} = E^\infty \sqcup (E^* \setminus E^*A) = G(E)^{(0)} \setminus E^* A$. Then $G(E,A)^{(0)}$ is a closed invariant subset of $G(E)^{(0)}$. Now define $G(E,A) = G(E)|_{G(E,A)^{(0)}}$. Then there is an exact sequence
\[
0 \to C^*(G(E)|_{E^* A}) \to C^*(G(E)) \xrightarrow{\pi} C^*(G(E,A)) \to 0.
\]

\begin{Theorem} \label{thm relative toeplitz algebra as groupoid algebra}

$C^*(G(E,A)) \cong \CT C^*(E,A)$ (as in Definition \ref{def relative toeplitz algebra}.)

\end{Theorem}

\begin{proof}
Suppose that $\{P_v : v \in E^0 \} \cup \{S_e : e \in E^1 \}$ is a family in a $C^*$-algebra $B$ satisfying the relations for $\CT C^*(E,A)$. Then it also satisfies the relations for $\CT C^*(E)$, hence there is a $*$-homomorphism $\sigma : \CT C^*(E) = C^*(G(E)) \to B$ with $\sigma(p_v) = P_v$ and $\sigma(s_e) = S_e$. Then for $v \in A$, $\sigma(\chi_{[v,v,\{v\}]}) = P_e - \sum_{e \in v E^1} S_e S_e^* = 0$. Therefore $C^*(G(E)|_{E^* A}) \subseteq \ker \sigma$, and hence $\sigma$ factors through $C^*(G(E,A))$.

Conversely, suppose that $\sigma : C^*(G(E,A)) \to B$ is a $*$-homomorphism. Then $\sigma \circ \pi : C^*(G(E)) = \CT C^*(E) \to B$ is a $*$-homomorphism. Thus the elements $\sigma(\pi(p_v))$ and $\sigma(\pi(s_e))$ satisfy the relations $(CK1)-(CK3)$. Since $C^*(G(E)|_{E^* A}) \subseteq \ker \sigma \circ \pi$, these elements satisfy (CK4) at vertices $v \in A$. Thus $\sigma$ is a representation of $\CT C^*(E,A)$.
\end{proof}

For the rest of this section we use the following

\begin{Notation}

Let $E$ be a graph, $H \subseteq E^0$ a hereditary subset, and let $F = F_H$. Let $A \subseteq \reg{E}$ and $B \subseteq \reg{F}$ be subsets.

\end{Notation}

\begin{Lemma} \label{lem F closed in E}

$G(F)^{(0)}$ is a closed invariant subset of $G(E)^{(0)}$.

\end{Lemma}

\begin{proof}
We note that
\begin{align}
E^\infty &= E^* H E^\infty \sqcup E^* F^\infty = E^* H E^\infty \sqcup F^\infty \notag \\
E^* &= E^* H \sqcup E^* F^0 = E^* H \sqcup F^* \label{eqn decompositions} \\
G(E)^{(0)} &= E^* H E^\infty \sqcup E^* H \sqcup G(F)^{(0)}. \notag
\end{align}
We will show that $G(E)^{(0)} \setminus G(F)^{(0)}$ is open. First let $x \in E^* H E^\infty$. Then $x = \alpha y$ for some $\alpha \in E^*$ and $y \in E^\infty$ such that $s(\alpha) = r(y) \in H$. Then $x \in Z(\alpha) \subseteq E^* H \cup E^* H E^\infty = (G(F)^{(0)})^c$, hence $Z(\alpha)$ is a neighborhood of $x$ disjoint from $G(F)^{(0)}$. Next let $\alpha \in E^* H$. Then again $Z(\alpha) \subseteq (G(F)^{(0)})^c$. Invariance is clear.
\end{proof}

Note that $G(F,B)^{(0)}$ is a closed subset of $G(F)^{(0)}$, and hence is a closed invariant subset of $G(E)^{(0)}$. 

\begin{Theorem}

$G(F,B)^{(0)} \subseteq G(E,A)^{(0)}$ if and only if $A \cap F^0 \subseteq B$.

\end{Theorem}

\begin{proof}
\begin{align*}
G(F,B)^{(0)} &= F^\infty \cup F^*(F^0 \setminus B)
= F^\infty \sqcup E^*(F^0 \setminus B) \\
G(E,A)^{(0)} &= E^\infty \sqcup E^*(E^0 \setminus A).
\end{align*}
Since $F^\infty \subseteq E^\infty$, $G(F,B)^{(0)} \subseteq G(E,A)^{(0)}$ if and only if $F^0 \setminus B \subseteq E^0 \setminus A$, or equivalently, if and only if $A \cap F^0 \subseteq B$.
\end{proof}

Assume that $A \cap F^0 \subseteq B$. Let $U = G(E,A)^{(0)} \setminus G(F,B)^{(0)}$, an open $G(E,A)$-invariant subset. We have an exact sequence
\[
0 \to C^*(G(E,A)|_U) \to \CT C^*(E,A) \to \CT C^*(F,B) \to 0.
\]
We will describe the ideal in this sequence by means of a generating family of projections, as is done traditionally (cf. \cite{bhrs}).

\begin{Theorem} \label{thm larger ideal}

Assume that $A \cap F^0 \subseteq B$. Let $U = G(E,A)^{(0)} \setminus G(F,B)^{(0)}$. Then $C^*(G(E,A)|_U)$ is the ideal in $\CT C^*(E,A)$ generated by $\{p_v : v \in H \} \cup \{ p_v - p_{v,H} : v \in B \setminus A \}$.

\end{Theorem}

\begin{proof}
We begin by dissecting $U$. Using \eqref{eqn decompositions} we have
\begin{align*}
G(E,A)^{(0)}
&= E^\infty \sqcup E^*(E^0 \setminus A) \\
&= E^* H E^\infty \sqcup F^\infty \sqcup E^* (H \setminus A) \sqcup F^*(F^0 \setminus A) \\
&= E^* H E^\infty \sqcup E^* (H \setminus A) \sqcup F^\infty \sqcup F^*(F^0 \setminus B) \sqcup F^*(B \setminus A) \\
&= E^* H E^\infty \sqcup E^* (H \setminus A) \sqcup F^*(B \setminus A) \sqcup G(F,B)^{(0)}, \\
\noalign{hence}
U
&= E^* H E^\infty \sqcup E^* (H \setminus A) \sqcup F^*(B \setminus A).
\end{align*}
We will write this as $U = \bigsqcup_{j=1}^3 U_j$.

Now we consider the ideal generated as in the statement of the theorem. For $v \in H$ we have that $p_v = \chi_{Z(v) \cap G(E,A)^{(0)}}$. We observe that
\[
Z(v) \cap G(E,A)^{(0)}
= v E^\infty \sqcup v E^*(E^0 \setminus A) 
= v E^\infty \sqcup v E^*(H \setminus A).
\]
The invariant subset of $G(E,A)^{(0)}$ generated by $\{ Z(v) \cap G(E,A)^{(0)} : v \in H \}$ is then
\[
W_1 :=
 \bigcup_{v \in H} E^* (Z(v) \cap G(E,A)^{(0)})
= \bigcup_{v \in H} (E^* v E^\infty \cup E^* v E^* (H \setminus A))
= E^* H E^\infty \cup E^*(H \setminus A).
\]
For $v \in B \setminus A$, $p_v - p_{v,H} = \chi_{(Z(v) \setminus \cup_{e \in v F^1} Z(e)) \cap G(E,A)^{(0)}}$. We have
\begin{align*}
\bigl(Z(v) \setminus \cup_{e \in v F^1} Z(e)\bigr) \cap G(E,A)^{(0)}
&= \bigl(\{v\} \cup \bigcup_{e \in v E^1 H} Z(e)\bigr) \cap G(E,A)^{(0)} \\
&= \{v\} \cup v E^1 H E^\infty \cup v E^1 H E^* (H \setminus A) \\
&\subseteq \{v\} \cup E^* H E^\infty \cup E^* (H \setminus A).
\end{align*}
The invariant subset of $G(E,A)^{(0)}$ generated by $\{ (Z(v) \setminus \cup_{e \in v F^1} Z(e)) \cap G(E,A)^{(0)} : v \in B \setminus A \}$ is then
\begin{align*}
W_2 &:=
\bigcup_{v \in B \setminus A} (E^*v \cup E^* v E^1 H E^\infty \cup E^* v E^1 H E^* (H \setminus A)) \\
&= E^*(B \setminus A) \cup W_3,
\end{align*}
where $W_3 \subseteq W_1$. Therefore the ideal in $\CT C^*(E,A)$ generated by $\{ p_v : v \in H \} \cup \{p_v - p_{v,H} : v \in B \setminus A \}$ is the ideal generated by $C_0(W_1 \cup W_2)$, and we have that $W_1 \cup W_2 = E^* H E^\infty \cup E^*(H \setminus A) \cup F^*(B\setminus A) = U$.
\end{proof}

\begin{Definition} \label{def J(E,A;F,B)}

We denote by $J(E,A;F,B)$ the ideal $C^*(G(E,A)|_U)$ in Theorem \ref{thm larger ideal}.

\end{Definition}

\begin{Corollary} \label{cor kernel of quotient map}

There is a (necessarily surjective) homomorphism $\CT C^*(E,A) \to \CT C^*(F,B)$ determined by $p_v \mapsto 0$ for $v \in H$, $s_e \mapsto s_e$ for $e \in F^1$ and $s_e \mapsto 0$ if $e \not\in F^1$ if and only if $A \cap F^0 \subseteq B$. In this case, the kernel of the homomorphism is $J(E,A;F,B)$. If $A = \reg{E}$ and $A \cap F^0 \subseteq B$, then $\reg{E} \cap F^0 \subseteq \reg{F}$. Then $\reg{E} \cap \source{F} = \varnothing$, and it follows that $H$ is saturated (see Remark \ref{remark saturated}). In this case, $\CT C^*(E,A) = C^*(E)$, $\CT C^*(F,B) = C^*(F)$, and the kernel is $J(E,\reg{E};F,\reg{F}) = J_{H,B_H}$ (as in \cite{bhrs}.)

\end{Corollary}

\section{Gauge-invariant ideals} \label{section gauge-invariant ideals}

The ideals in $\CT C^*(E)$ (and in $\CT C^*(E,A)$) that we considered in section \ref{section groupoid picture}  are the \textit{gauge-invariant} ideals. The gauge-invariant ideals in $C^*(E)$ were completely described in \cite{bhrs}. This was extended to the case of higher rank graphs, and to the Toeplitz algebras $\CT C^*(\Lambda)$, in \cite{sww}. In the groupoid picture these are the ideals that arise from open invariant subsets of $G(E)^{(0)}$ (see Theorem \ref{thm gauge-invariant ideals}), and such sets can be described in terms of the underlying graph. However the description in the generality of higher rank graphs is extremely complicated. A simple proof was given for the case of directed graphs in \cite{ek}, using the results in \cite{bhrs} for Cuntz-Krieger algebras. In this section we give an alternative proof that bypasses the theorem of \cite{bhrs}. Much of the argument can be given in the more general context of a $C^*$-algebra with an action of a compact abelian group, or of the $C^*$-algebra of an \'etale groupoid. We hope that our arguments may prove useful in more general situations than that of higher-rank graphs. We begin with some well-known facts about fixed point algebras of compact abelian group actions. Let $A$ be a $C^*$-algebra, let $K$ be a compact abelian group, and let $\gamma : K \to \text{Aut}(A)$ be a (point-norm) continuous action. Define $P : A \to A$ by $P(x) = \int_K \gamma_t(x) dt$. Then $P$ is a faithful conditional expectation onto the fixed point algebra $A^\gamma := \{ x \in A : \gamma_t(x) = x,\ t \in K \}$. We will use the following notation: if $A$ is a $C^*$-algebra and $S \subseteq A$ we write $\langle S \rangle_A$ for the ideal in $A$ generated by $S$.

\begin{Lemma} \label{lem fixed point algebra}

If $I$ is a $\gamma$-invariant ideal in $A$ then $P(I) = I \cap A^\gamma$, and $I = \langle I \cap A^\gamma \rangle_A$.

\end{Lemma}

\begin{proof}
Let $I$ be a $\gamma$-invariant ideal in $A$. Then $P(I) \subseteq I$, hence $P(I) \subseteq I \cap A^\gamma$. Conversely, since $P|_{A^\gamma}$ is the identity map, $I \cap A^\gamma = P(I \cap A^\gamma) \subseteq P(I)$. Therefore $P(I) = I \cap A^\gamma$.

In particular, $P(I)$ is an ideal in $A^\gamma$. Let $I' = \langle P(I) \rangle_A$. Since $P(I)$ is pointwise fixed by $\gamma$, hence is $\gamma$-invariant, $I'$ is also $\gamma$-invariant. Since $P(I) \subseteq I$, it follows that $I' \subseteq I$. Then $P(I') \subseteq P(I) \subseteq I' \cap A^\gamma = P(I')$, so $P(I') = P(I)$. We claim that $I' = I$. For if not, then $I'$ is a proper ideal in $I$, and there exists $x \ge 0$, $x \in I \setminus I'$. Let $\pi_0 : A \to A/I'$ be the quotient map. Then $\pi_0(x) \not= 0$. We may define $\beta : K \to \text{Aut}(A/I')$ by $\beta_t(a + I') = \gamma_t(a) + I'$, and let $P' : A/I' \to (A/I')^\beta$ the corresponding conditional expectation. Then $P'(\pi_0(x)) \not= 0$, but $P'(\pi_0(x)) = \pi_0(P(x)) = 0$, since $P(x) \in P(I) = P(I') \subseteq I'$. Therefore $I' = I$.
\end{proof}

In this paper we are concerned with invariant ideals in the $C^*$-algebras of \'etale groupoids. In particular, we will give conditions under which these ideals correspond to open groupoid-invariant subsets of the unit space. We cite the following elementary fact. A proof is given in \cite[Lemma 3.2]{bonli} for the reduced $C^*$-algebra (which will be sufficient for our use since we will consider amenable groupoids). However that proof still holds for the full $C^*$-algebra.

\begin{Lemma} \label{lem ideals and invariant open sets}

Let $G$ be a Hausdorff \'etale groupoid, let $U \subseteq G^{(0)}$ be an open $G$-invariant set, and let $I$ be the ideal in $C^*(G)$ generated by $C_0(U)$. Then $I \cap C_0(G^{(0)}) = C_0(U)$.

\end{Lemma}

A common instance of a compact abelian group action occurs when a groupoid has a continuous homomorphism into a discrete abelian group (such a homomorphism is often called a \textit{cocycle}). Suppose that $G$ is a Hausdorff \'etale groupoid with a continuous homomorphism $c : G \to \Gamma$, where $\Gamma$ is a discrete abelian group. There is a dual action $\gamma$ of $\widehat{\Gamma}$ on $C^*(G)$, defined on $C_c(G)$ by $\gamma_z(f) (g) = \langle z,c(g) \rangle f(g)$. Let $G^c = c^{-1}(0)$. Then $G^c$ is also a Hausdorff \'etale groupoid, with the same unit space as $G$. It is shown in \cite[Proposition 9.3]{spi2} that if $G^c$ is amenable then so is $G$, and in fact $C^*(G)^\gamma = C^*(G^c)$. There are many situations where $G^c$ is an AF groupoid, and hence amenable (\cite[Definition III.1.1]{ren1}); some examples are the $C^*$-algebras of higher rank graphs, and more generally, of certain categories of paths (\cite[Theorem 9.8]{spi2}), and also the Toeplitz Cuntz-Krieger algebras of directed graphs. We give the following result for these situations. 

\begin{Lemma} \label{lem AF fixed point groupoid}

Let $G$ be a Hausdorff ample groupoid with a continuous homomorphism $c : G \to \Gamma$, for some discrete abelian group $\Gamma$. Suppose that the fixed point subgroupoid $G^c$ is an AF groupoid. Let $\gamma$ be the dual action of $K = \widehat{\Gamma}$ on $C^*(G)$. Let $I$ be a $\gamma$-invariant ideal in $C^*(G)$, and let $I \cap C_0(G^{(0)}) = C_0(U)$, where $U$ is an open $G$-invariant subset of $G^{(0)}$. Then $I = \langle C_0(U) \rangle_{C^*(G)}$.

\end{Lemma}

\begin{proof}
It is clear that $I \supseteq \langle C_0(U) \rangle_{C^*(G)}$. Before proving the reverse containment, we will prove the following claim: $I \cap C^*(G)^\gamma = \langle C_0(U) \rangle_{C^*(G^c)}$. Since $C_0(U) \subseteq I$ it is clear that $I \cap C^*(G)^\gamma \supseteq \langle C_0(U) \rangle_{C^*(G^c)}$.  For the reverse containment we use the facts that $C^*(G)^\gamma = C^*(G^c)$ and that $G^c$ is an AF groupoid. As in the proof of \cite[Proposition III.1.15]{ren1}, we may write $C^*(G^c) = \overline{\cup_n B_n}$ where $B_n \subseteq B_{n+1}$ are finite dimensional $C^*$-algebras, and we may write $C_0(G^{(0)}) = \overline{\cup_n D_n}$ where $D_n \subseteq D_{n+1}$ and $D_n$ is a maximal abelian subalgebra (masa) of $B_n$. By \cite[Lemma 3.1]{bra} we have that $I \cap C^*(G^c) = \overline{\cup_n C_n}$, where $C_n \subseteq C_{n+1}$ and $C_n$ is an ideal in $B_n$. Then $C_0(U) = I \cap C_0(G^{(0)}) = \overline{\cup_n (C_n \cap D_n)}$. Note that $C_n = \langle C_n \cap D_n \rangle_{B_n}$. Then
\[
I \cap C^*(G^c) = \overline{\cup_n C_n}
= \overline{\cup_n \langle C_n \cap D_n \rangle_{B_n}}
\subseteq \langle I \cap C_0(G^{(0)}) \rangle_{C^*(G^c)}
= \langle C_0(U) \rangle_{C^*(G^c)}.
\]
This finishes the proof of the claim. Now we have
\begin{align*}
I &= \langle I \cap C^*(G^c) \rangle_{C^*(G)}, \text{ by Lemma \ref{lem fixed point algebra},} \\
&= \bigl\langle \langle C_0(U) \rangle_{C^*(G^c)} \bigr\rangle_{C^*(G)} \\
&= \langle C_0(U) \rangle_{C^*(G)}. \qedhere
\end{align*}
\end{proof}

\begin{Corollary} \label{cor AF fixed point groupoid}

In the context of Lemma \ref{lem AF fixed point groupoid}, the map $U \mapsto \langle C_0(U) \rangle_{C^*(G)}$ is a bijection between the open $G$-invariant subsets of $G^{(0)}$ and the $\gamma$-invariant ideals of $C^*(G)$.

\end{Corollary}

\begin{proof}
This follows from Lemmas \ref{lem ideals and invariant open sets} and \ref{lem AF fixed point groupoid}.
\end{proof}

We recall the \textit{gauge action} on the Toeplitz algebra of a directed graph $E$. There is an action $\gamma$ of the circle group $\IT$ on $\CT C^*(E)$ defined by $\gamma_z(p_v) = p_v$ for $v \in E^0$, and $\gamma_z(s_e) = z s_e$ for $e \in E^1$. (This is the dual action (as discussed before Lemma \ref{lem AF fixed point groupoid}) to the \textit{length cocycle} $c : G(E) \to \IZ$, defined by $c([\alpha,\beta,x]) = |\alpha| - |\beta|$.) By Corollary \ref{cor AF fixed point groupoid} we know that the gauge-invariant ideals of $\CT C^*(E)$ are in one-to-one correspondence with the open invariant subsets of $G(E)^{(0)}$. We now describe the open invariant subsets of $G(E)^{(0)}$.

\begin{Lemma} \label{lem open invariant sets one}

Let $H \subseteq E^0$ be a hereditary subset and let $F = F_H$ as in Definition \ref{def hereditary}. Let $B \subseteq \reg{F}$. Let $U(H,B):= E^* H E^\infty \cup E^*(H \cup B)$. Then $U(H,B)$ is open and invariant. Moreover, $H = \{ v \in E^0 : Z(v) \subseteq U(H,B) \}$ and $B = F^0 \cap U(H,B)$.

\end{Lemma}

\begin{proof}
It is clear from its definition that $U(H,B)$ is invariant. We show that it is open. Let $x \in U(H,B)$. First suppose that $x \in E^\infty$. Then $x \in E^* H E^\infty$, so we may write $x = \alpha y$ where $s(\alpha) \in H$. Then $Z(s(\alpha)) \subseteq H E^\infty \cup H E^* \subseteq U(H,B)$. Therefore $Z(\alpha) = \alpha Z(s(\alpha)) \subseteq U(H,B)$ is a neighborhood of $x$. Next suppose that $x \in E^* H$. Then $s(x) \in H$, so as before, $x Z(s(x))$ is a neighborhood of $x$ contained in $U(H,B)$. Lastly suppose that $x \in E^* B$. Then $s(x) \in B \subseteq \reg{F}$, hence $s(x) F^1$ is finite. For $e \in s(x) E^1 H$, $Z(s(e)) \subseteq U(H,B)$. Then
\[
Z(s(x)) \setminus \bigcup_{e \in s(x) F^1} Z(e)
= \{ s(x) \} \cup \bigcup_{e \in s(x) E^1 H} Z(e)
= \{ s(x) \} \cup \bigcup_{e \in s(x) E^1 H} e Z(s(e))
\subseteq U(H,B).
\]
Therefore $Z(x) \setminus \bigcup_{e \in s(x) F^1} Z(xe)$ is a neighborhood of $x$ in $U(H,B)$.

For the last statement, it is clear that $H \subseteq \{ v : Z(v) \subseteq U(H,B) \}$ and $B \subseteq F^0 \cap U(H,B)$. We prove the reverse inclusions. First let $v \in E^0$ with $Z(v) \subseteq U(H,B)$. Then $v \in U(H,B)$, so $v \in H \cup B$. To show that $v \in H$, we suppose that $v \in B$ and deduce a contradiction. We know that $v E^* \subseteq Z(v) \subseteq U(H,B)$, so $s(v E^*) \subseteq H \cup B$. Note that since $B \subseteq \reg{F}$, if $\alpha \in v E^* B$ then $s(\alpha) F^1 \not= \varnothing$. Applying this to $v$ we find $e_1 \in v F^1$. Applying it to $e_1$ we find $e_2 \in s(e_1) F^1$. Inductively we obtain $x = e_1 e_2 \cdots \in v F^\infty$. Since $x \in Z(v) \subseteq U(H,B)$, this is impossible - all infinite paths in $U(H,B)$ arise in $H$. Therefore $v \in H$ as claimed. Finally, suppose that $v \in F^0 \cap U(H,B)$. Again we must have $v \in H \cup B$. Since $v \in F^0$ it follows that $v \in B$.
\end{proof}

\begin{Lemma} \label{lem open invariant sets two}

Let $U \subseteq G(E)^{(0)}$ be open and invariant. Let $H := \{ v \in E^0 : Z(v) \subseteq U \}$ and $B = U \cap (E^0 \setminus H)$. Then $H$ is hereditary in $E$. Letting $F = F_H$, we have that $B \subseteq \reg{F}$. 
Finally, $U = U(H,B)$ (where $U(H,B)$ is as in the statement of Lemma \ref{lem open invariant sets one}).

\end{Lemma}

\begin{proof}
We begin by showing that $H$ is hereditary. Let $\alpha \in H E^*$. Then $Z(r(\alpha) \subseteq U$. Since $Z(\alpha) \subseteq Z(r(\alpha))$ we have that $Z(\alpha) \subseteq U$. But $Z(\alpha) = \alpha Z(s(\alpha))$. Since $U$ is invariant, $Z(s(\alpha)) \subseteq U$, hence $s(\alpha) \in H$. Therefore $H$ is hereditary.

Next we show that $B \subseteq \reg{F}$. Let $v \in B$. Then
\[
Z(v)
= \{ v \} \cup \bigcup_{e \in v E^1} Z(e)
= \{ v \} \cup \bigcup_{e \in v F^1} Z(e) \cup \bigcup_{e \in v E^1 H} e Z(s(e)).
\]
Since $Z(w) \subseteq U$ for $w \in H$, the invariance of $U$ implies that $\{v\} \cup \bigcup_{e \in v E^1 H} e Z(s(e)) \subseteq U$. Since $Z(v) \not\subseteq U$ it follows that $v F^1 \not= \varnothing$. Since $U$ is open there are $\alpha_1,\ldots,\alpha_n \in v E^*$ such that $Z(v) \setminus \bigcup_{i=1}^n Z(\alpha_i) \subseteq U$. Let $e_1,\ldots,e_n \in E^1$ be such that $\alpha_i \in Z(e_i)$ for $1 \le i \le n$. For $e \in v E^1 \setminus \{ e_1,\ldots,e_n \}$ we have $Z(e) \subseteq Z(v) \setminus Z(\alpha_i) \subseteq U$, hence $s(e) \in H$. Therefore $e \not\in F^1$. Thus $v F^1 \subseteq \{e_1,\ldots,e_n\}$ is finite. It follows that $v \in \reg{F}$.

Finally we show that $U(H,B) = U$. The invariance of $U$ implies that $U(H,B) \subseteq U$. We prove the reverse containment. Let $x \in U$. Since $U$ is open there are $\alpha,\beta_1,
\ldots,\beta_n \in r(x) E^*$ such that $x \in Z(\alpha) \setminus \bigcup_{i=1}^n Z(\beta_i) \subseteq U$. If $x \in E^\infty$ we may choose $\gamma \in r(x) E^*$ such that $x \in Z(\gamma)$ and $|\gamma| > |\beta_i|$ for $1 \le i \le n$. Then $Z(\gamma) \cap Z(\beta_i) = \varnothing$ for $1 \le i \le n$, so $Z(\gamma) \subseteq U$. But then $s(\gamma) \in H$, so $x \in E^* H E^\infty \subseteq U(H,B)$. If $x \in E^*$ then since $U$ is invariant, $s(x) \in U$. Then $s(x) \in H \cup B$, so $x \in E^*(H \cup B) \subseteq U(H,B)$.
\end{proof}

We summarize the above:

\begin{Theorem} \label{thm gauge-invariant ideals}

Let $E$ be a directed graph with groupoid $G(E)$. The gauge-invariant ideals in $\CT C^*(E)$ are in one-to-one correspondence with the open $G(E)$-invariant subsets of $G(E)^{(0)}$, where an open invariant set $U$ corresponds to the ideal generated by $C_0(U)$. The open invariant subsets of $G(E)^{(0)}$ are in one-to-one correspondence with pairs $(H,B)$, where $H \subseteq E^0$ is a hereditary subset, and $B \subseteq \reg{F_H}$. The open invariant set corresponding to the pair $(H,B)$ is $U(H,B) = E^* H E^\infty \cup E^* (H \cup B)$.

\end{Theorem}

\begin{Remark}

In Theorem \ref{thm gauge-invariant ideals}, the gauge-invariant ideal corresponding to the open invariant set $U(H,B)$ is $J(E,\varnothing;F,B)$ (as in Definition \ref{def J(E,A;F,B)}).

\end{Remark}

\begin{Remark}

Our classification of the gauge invariant ideals in $\CT C^*(E)$ is equivalent to that given in \cite[Theorem 2.11]{ek} by hereditary and partially saturated sets.

\end{Remark}

\section{Relative graphs} \label{section relative graphs}

\vspace*{.1 in}

\begin{Definition} \label{definition relative graph}

A \textit{relative graph} is a pair $(F,B)$ consisting of a directed graph $F$ and a subset $B$ of the regular vertices of $F$. A \textit{morphism} $\alpha : (F,B) \to (E,A)$ of relative graphs is an injective homomorphism of directed graphs $F \hookrightarrow E$, which we take to be an inclusion, such that

\begin{enumerate}[(1)]

\item \label{relative graph 1} $H_{F,E} := E^0 \setminus F^0$ is a hereditary set in $E$

\item \label{relative graph 2} $F^1 = F^0  E^1 F^0 \ (= \{e \in E^1 : s(e), r(e) \in F^0 \})$.

\item \label{relative graph 3} $A \cap F^0 \subseteq B$ 

\end{enumerate}

\end{Definition}

This definition describes the situation where $\CT C^*(F,B) \cong \CT C^*(E,A) / J(E,A;F,B)$, by Theorem \ref{thm larger ideal}. (Note that $F_{H_{F,E}} = F$.)

\begin{Lemma} \label{lemma composition}

Let $\alpha_1 : (F_1,A_1) \to (F_2,A_2)$ and $\alpha_2 : (F_2,A_2) \to (F_3,A_3)$ be morphisms of relative graphs. Then $\alpha_2 \circ \alpha_1 : (F_1,A_1) \to (F_3,A_3)$ is a morphism of relative graphs.

\end{Lemma}

\begin{proof}
We verify the conditions for a morphism in Definition \ref{definition relative graph}. Since $F_1 \hookrightarrow F_2$ and $F_2 \hookrightarrow F_3$ we have $F_1 \hookrightarrow F_3$. For convenience we will write $H_{ij}$ for $H_{F_i,F_j}$. We first show that $H_{13}$ is hereditary in $F_3$. Thus we must show that $H_{13} F_3^1 = H_{13} F_3^1 H_{13}$. Note that $H_{13} = F_3^0 \setminus F_1^0 = (F_3^0 \setminus F_2^0) \sqcup (F_2^0 \setminus F_1^0) = H_{23} \sqcup H_{12}$. Then $H_{13} F_3^1 = H_{23} F_3^1 \sqcup H_{12} F_3^1$. Since $\alpha_2$ is a morphism we know that $H_{23}$ is hereditary in $F_3$. Therefore $H_{23} F_3^1 = H_{23} F_3^1 H_{23} \subseteq H_{13} F_3^1 H_{13}$. Next, since $F_3^0 = H_{13} \sqcup F_1^0$, we have $H_{12} F_3^1 =  H_{12} F_3^1 H_{13} \sqcup H_{12} F_3^1 F_1^0$. But $H_{12}F_3^1 F_1^0 \subseteq F_2^0 F_3^1 F_2^0 = F_2^1$, by \eqref{relative graph 2} for $\alpha_2$, and hence $H_{12} F_3^1 F_1^0 \subseteq H_{12} F_2^1 F_1^0 = \varnothing$, since $H_{12}$ is hereditary in $F_2$. Therefore $H_{12} F_3^1 \subseteq H_{13} F_3^1 H_{13}$, and hence $H_{13}$ is hereditary in $F_3$. This finishes the verification of \eqref{relative graph 1} for $\alpha_2 \circ \alpha_1$.

Next, note that $F_1^0 F_3^1 F_1^0 \subseteq F_2^0 F_3^1 F_2^0 = F_2^1$ by \eqref{relative graph 2} for $\alpha_2$. Then $F_1^0 F_3^1 F_1^0 = F_1^0 F_2^1 F_1^0 = F_1^1$ by \eqref{relative graph 2} for $\alpha_1$. Thus \eqref{relative graph 2} holds for $\alpha_2 \circ \alpha_1$.

Finally, to verify \eqref{relative graph 3}, note that $A_3 \cap F_1^0 \subseteq A_3 \cap F_2^0 \subseteq A_2$, hence $A_3 \cap F_1^0 = A_3 \cap A_2 \cap F_1^0 \subseteq A_2 \cap F_1^0 \subseteq A_1$.
\end{proof}

\begin{Theorem} \label{thm pushout}

The category of relative graphs admits pushouts.

\end{Theorem}

\begin{proof}
Let $\alpha_i : (F_0,A_0) \to (F_i,A_i)$, $i = 1$, 2, be morphisms of relative graphs. Define a relative graph $(E,A)$ by
\begin{align}
E &= F_1 \sqcup_{F_0} F_2 \label{pushout eqn one} \\
A &= (A_1 \setminus F_0^0) \cup (A_2 \setminus F_0^0) \cup (A_1 \cap A_2). \label{pushout eqn two}
\end{align}
(The graph $E$ is defined by $E^0 = (F_1^0 \sqcup F_2^0)/\{\alpha_1(v) = \alpha_2(v) : v \in F_0^0 \}$, $E^1 = (F_1^1 \sqcup F_2^1)/\{\alpha_1(e) = \alpha_2(e) : e \in F_0^1 \}$, $r_E(e) = r_{F_i}(e)$ for $e \in F_i^1$, and $s_E$ defined analogously.) We note that $A_1 \cap A_2 \subseteq F_1^0 \cap F_2^0 = F_0^0$. We further note that
\[
H_{F_0,F_1}
= F_1^0 \setminus F_0^0
= (F_1^0 \sqcup (F_2^0 \setminus F_0^0)) \setminus (F_0^0 \sqcup (F_2^0 \setminus F_0^0))
= E^0 \setminus F_2^0
= H_{F_2,E},
\]
and similarly, $H_{F_0,F_2} = H_{F_1,E}$. We claim that $H_{F_2,E}$ is hereditary in $E$. To see this let $e \in H_{F_2,E} E^1$. Then $r(e) \in H_{F_0,F_1} = F_1^0 \setminus F_0^0 = F_1^0 \setminus F_2^0$. But then $e \not\in F_2^1$, hence $e \in F_1^1$. Since $H_{F_0,F_1}$ is hereditary in $F_1$ we have that $s(e) \in H_{F_0,F_1} = H_{F_2,E}$. This proves the claim. Similarly, $H_{F_1,E}$ is hereditary in $E$.

We now show that $A \subseteq \reg{E}$. Let $v \in A$. First suppose that $v \in A_1 \setminus F_0^0 \subseteq \reg{F_1} \cap H_{F_0,F_1}$. Since $H_{F_0,F_1}= H_{F_2,E}$ is hereditary in $E$, $v E^1 = v E^1 H_{F_2,E} \subseteq v E^1 (E^0 \setminus F_2^0) \subseteq v E^1 F_1^0 \subseteq v F^1$. Then since $v \in \reg{F_1}$, $v \in \reg{E}$. An analogous argument shows that $A_2 \setminus F_0^0 \subseteq \reg{E}$. Finally, $A_1 \cap A_2 \subseteq \reg{F_1} \cap \reg{F_2} \subseteq \reg{E}$. Therefore $A \subseteq \reg{E}$.

Let $\beta_i : (F_i,A_i) \to (E,A)$ for $i = 1$, 2. We claim that $\beta_1$ and $\beta_2$ are morphisms of relative graphs. For the proof, first note that we have already shown that $H_{F_i,E}$ is hereditary in $E$ for $i = 1,2$. Equation \eqref{relative graph 3} for $\beta_1$ and $\beta_2$ is immediate from the definition of $A$. For \eqref{relative graph 2}, we have \begin{align*}
F_i^0 E^1 F_i^0
&= (F_0^0 \sqcup H_{F_0,F_i}) E^1 (F_0^0 \sqcup H_{F_0,F_i}) \\
&= F_0^0 E^1 F_0^0 \sqcup F_0^0 E^1 H_{F_0,F_i} \sqcup H_{F_0,F_i} E^1 F_0^0 \sqcup H_{F_0,F_i}E^1 H_{F_0,F_i} \\
&= F_0^1 \sqcup F_0 F_i^1 H_{F_0,F_i} \sqcup \varnothing \sqcup H_{F_0,F_i} F_i^1 \\
&= F_i^1.
\end{align*}
It is clear that $\beta_1 \alpha_1 = \beta_2 \alpha_2$.

Now let $(G,B)$ be a relative graph and let $\gamma_i : (F_i,A_i) \to (G,B)$ be morphisms such that $\gamma_1 \alpha_1 = \gamma_2 \alpha_2$. Then the inclusions $F_i \hookrightarrow G$ agree on $F_0$, so there is an inclusion $E \hookrightarrow G$. We claim that $\phi : (E,A) \to (G,B)$ is a morphism. It is clear that then $\gamma_i = \phi \beta_i$ for $i = 1$, 2. We verify \eqref{relative graph 1} - \eqref{relative graph 3} for $\phi$. For \eqref{relative graph 3} first note that $B \cap E^0 = (B \cap F_1^0) \cup (B \cap F_2^0) \subseteq A_1 \cup A_2$. Next we have $B \cap F_0^0 = B \cap F_1^0 \cap F_2^0 \subseteq A_1 \cap A_2$. Then we have
\[
B \cap E^0
= (B \cap E^0 \setminus F_0^0) \cup (B \cap F_0^0)
\subseteq ((A_1 \cup A_2) \setminus F_0^0) \cup (A_1 \cap A_2 \cap F_0^0) = A.
\]
To prove \eqref{relative graph 2}, note that from \eqref{relative graph 2} for $\gamma_1$ and $\gamma_2$ we have that $F_i^0 G^1 F_i^0 = F_i^1$ for $i = 1$, 2. Then
\[
E^0 G^1 E^0 = (F_1^0 \cup F_2^0) G^1 (F_1^0 \cup F_2^0)
\subseteq E^1 \cup F_1^0 G^1 F_2^0 \cup F_2^0 G^1 F_1^0.
\]
We claim that $F_1^0 G^1 F_2^0$, $F_2^0 G^1 F_1^0 \subseteq E^1$. For suppose, say, that $e \in F_1^0 G^1 F_2^0$. If $r(e) \in F_0^0$ then $e \in F_2^0 G^1 F_2^0 = F_2^1 \subseteq E^1$, by \eqref{relative graph 2} for $\gamma_2$. Similarly, if $s(e) \in F_0$ then $e \in E^1$. If neither of these two situations occurs, then $e \in H_{F_0,F_1} G^1 H_{F_0,F_2}$. Then $r(e) \in F_1^0 \setminus F_0^0 \subseteq G^0 \setminus F_2^0 = H_{F_2,G}$, which is hereditary by \eqref{relative graph 1} for $\gamma_1$. Thus $s(e) \in H_{F_2,G}$, a set disjoint from $H_{F_0,F_2}$. Thus one of the previous two situations must occur, proving the claim. Therefore $E^0 G^1 E^0 = E^1$. This verifies \eqref{relative graph 2} for $\phi$.

For \eqref{relative graph 1}, note first that $H_{E,G} = G^0 \setminus (F_1^0 \cup F_2^0) = (G^0 \setminus F_1^0) \cap (G^0 \setminus F_2^0) = H_{F_1,G} \cap H_{F_2,G}$. Since $H_{F_i,G}$ is hereditary in $G$ for $i = 1$, 2, then $H_{F_1,G} \cap H_{F_2,G}$ is as well. Thus $H_{E,G}$ is hereditary in $G$, verifying \eqref{relative graph 1}. Therefore $\phi$ is a morphism.
\end{proof}

\begin{Remark} \label{rmk general situation}
Let us consider the data involved in a pushout of relative graphs from a larger perspective. Let $F_0 \hookrightarrow F_i$, $i=1,2$, be inclusions of graphs satisfying Definition \ref{definition relative graph}\eqref{relative graph 1} and \eqref{relative graph 2}. For $i = 1,2$ let $A_i \subseteq \reg{F_i^0}$ be arbitrary. Let $A_{12} = (A_1 \cup A_2) \cap F_0^0$. Then if $A_0 \subseteq \reg{F_0^0}$, Definition \ref{definition relative graph}\eqref{relative graph 3} implies that the maps $(F_0,A_0) \to (F_i,A_i)$ are morphisms of relative graphs if and only if $A_{12} \subseteq A_0$. Notice that the pushout, $(E,A)$, defined in the proof of Theorem \ref{thm pushout}, does not depend on the choice of $A_0$ satisfying $A_{12} \subseteq A_0 \subseteq \reg{F_0}$. 
\end{Remark}

\begin{Remark}
Consider the pushout as defined in the proof of Theorem \ref{thm pushout}. We note the following.
\begin{equation}
\label{main theorem remark one} H_{F_1,E} \cap H_{F_2,E} = \varnothing
\end{equation}
Equation \eqref{main theorem remark one} is true because $H_{F_1,E} \cap H_{F_2,E} = (E^0 \setminus F_2^0) \cap (E^0 \setminus F_1^0) = E^0 \setminus (F_2^0 \cup F_1^0) = \varnothing$. 
\end{Remark}

\section{Admissible pushouts of relative graphs} \label{section admissible pushouts of relative graphs}

In order to discuss pushouts of relative Toeplitz graph algebras we recall Pedersen's theorem characterizing pullbacks of $C^*$-algebras in the case relevant to this paper.

\begin{Theorem} \label{theorem pullback}

Let $I$ and $J$ be ideals in a $C^*$-algebra $A$, and consider the commuting square of quotient maps in Figure \ref{figure 1}. This is a pullback diagram if and only if $IJ = 0$.

\end{Theorem}

\begin{proof}
This follows easily from \cite[Proposition 3.1]{ped}, as we show here. According to this proposition, the diagram is a pullback if and only if the following three conditions are satisfied:

\begin{enumerate}[(i)]

\item $I \cap J = IJ = \{0\}$,

\item $q_I^{-1}(q_J(A/J)) = \pi(A)$,

\item $\pi_J(\ker \pi_I) = \ker q_J$.

\end{enumerate}

\noindent
Condition (ii) is true since all four maps are surjective. To see that condition (iii) is true, note that $\pi_J(\ker \pi_I) = \pi_J(I) = (I + J)/J = \ker q_J$. Therefore the diagram is a pullback if and only if (i) holds.
\end{proof}

\begin{figure}

\begin{tikzpicture}[scale=3]

\node (0_0) at (0,0) [rectangle] {$A/J$};
\node (0_1) at (0,1) [circle] {$A$};
\node (1_0) at (1,0) [rectangle] {$A/(I + J)$.};
\node (1_1) at (1,1) [circle] {$A/I$};

\draw[-latex,thick] (0_1) -- (0_0) node[pos=0.5, inner sep=0.5pt, left=1pt] {$\pi_J$};
\draw[-latex,thick] (0_1) -- (1_1) node[pos=0.5, inner sep=0.5pt, above=1pt] {$\pi_I$};
\draw[-latex,thick] (1_1) -- (1_0) node[pos=0.5, inner sep=0.5pt, right=1pt] {$q_I$};
\draw[-latex,thick] (0_0) -- (1_0) node[pos=0.5, inner sep=0.5pt, below=1pt] {$q_J$};

\end{tikzpicture}
  \captionof{figure}{}
  \label{figure 1}

\end{figure}

By the definition of the category of relative graphs, a pushout in this category determines, on the one hand, a commuting square of relative Toeplitz graph algebras (the outer square of Figure \ref{figure 3}), and on the other hand, a commuting diagram of quotient $C^*$-algebras (the upper left triangle of Figure \ref{figure 3}), where $I_j = \ker \pi_j$ for $j = 1,2$.
\begin{figure}[h]
\begin{tikzpicture}
\node (E-A) at (0,0) [rectangle] {$\CT C^*(E,A)$};
\draw (E-A) ++(5,0) node (F_1-A_1) {$\CT C^*(F_1,A_1)$};
\draw (E-A) ++(0,-3) node (F_2-A_2) {$\CT C^*(F_2, A_2)$};
\draw (E-A) ++(F_1-A_1) ++(F_2-A_2) node (F_0-A_0) {$\CT C^*(F_0,A_0)$};

\draw (E-A) ++(2.5,-1.5) node (quot) {$\displaystyle \frac{\CT C^*(E,A)}{I_1 + I_2}$};

\draw (F_0-A_0) ++(2.2,0) node (quot2) {$\cong \displaystyle \frac{\CT C^*(E,A)}{I_0}$.};

\begin{scope}[>=latex]
\draw[->, thick] (E-A) -- (F_1-A_1) node[pos=0.5, above] {$\pi_1$};
\draw[->, thick] (E-A) -- (F_2-A_2) node[pos=0.5, left] {$\pi_2$};
\draw[->, thick] (F_1-A_1) -- (F_0-A_0);
\draw[->, thick] (F_2-A_2) -- (F_0-A_0);

\draw[->, thick] (F_1-A_1) -- (quot);
\draw[->, thick] (F_2-A_2) -- (quot);
\end{scope}
\end{tikzpicture}
\captionof{figure}{}
\label{figure 3}
\end{figure}
In this section we first prove that this upper triangle is a pullback diagram of quotient $C^*$-algebras, using Theorem \ref{theorem pullback}. We then give a condition on the pushout that is necessary and sufficient for the inner and outer diagrams to coincide, so that the commuting square is in fact a pullback diagram of relative Toeplitz graph algebras. As in \cite{hrt} we call this condition \textit{admissibility}. Our condition is quite different from the conditions for admissibility in \cite{hrt}. In Example \ref{example compare} we give a detailed comparison in the setting of \cite{hrt}. Generally speaking the main reason behind our approach is that when arbitrary graphs are considered, the breaking vertices give rise to a rich family of gauge-invariant ideals (as described in Theorem \ref{thm larger ideal}). The problem of characterizing those pushout diagrams of graphs that determine pullback diagrams of graph algebras leads naturally to the same problem in the larger context of relative graphs and relative Toeplitz graph algebras. In this setting it seems to us more appropriate to make the requirements of admissibility from \cite{hrt} part of the structure of the category of relative graphs. Thus we use the term \textit{admissible} for the new feature needed due to the interaction between the hereditary sets and the sets used for the relativizations.

\begin{Notation} \label{main proofs notn}

For the rest of this section, let $\alpha_i : (F_0,A_0) \to (F_i,A_i)$ be morphisms of relative graphs, for $i = 1, 2$, and let $(E,A)$ be the pushout. Recall from Remark \ref{rmk general situation} the set $A_{12} = (A_1 \cup A_2) \cap F_0^0$, satisfying $A_{12} \subseteq A_0 \subseteq \reg{F_0^0}$. For the sake of notational convenience we will write $F_{12} := F_0$.
\begin{enumerate}[(1)]
    \item \label{ideal notn} Let $I_i = J(E,A;F_i,A_i)$ for $i = 0, 1, 2, 12$.
    \item \label{u notn} Let $U^{(i)} = G(E,A)^{(0)} \setminus G(F_i,A_i)^{(0)}$ for $i = 0, 1, 2, 12$.
    \item \label{u decomp notn} For each of these four open sets we will make use of the corresponding decomposition as the union of three subsets, as given in the proof of Theorem \ref{thm larger ideal}. That is, for $i = 0,1,2,12$ we have
\[
U^{(i)} = \bigsqcup_{j=1}^3 U^{(i)}_j \\
= E^* H_{F_i,E} E^\infty \sqcup E^*(H_{F_i,E} \setminus A) \sqcup E^*(A_i \setminus A).
\]
\end{enumerate}

\end{Notation}

\begin{Remark} \label{rmk U's and I's}
By \ref{main proofs notn}(\ref{ideal notn}), $\CT C^*(F_i, A_i) = \CT C^*(E,A)/I_i$ for $i=0,1,2,12$. By \ref{main proofs notn}(\ref{u notn}), $I_i$ is generated as an ideal (in $\CT C^*(E,A)$) by $C_0(U^{(i)})$ for $i=0,1,2$, and 12.
\end{Remark}

\begin{Proposition}\label{prop pullback}
$I_1 I_2 = 0,$ and thus by Theorem \ref{theorem pullback}, $\CT C^*(E,A)$ is the pullback of $\CT C^*(F_1,A_1)$ and $\CT C^*(F_2,A_2)$ over $\CT C^*(E,A)/(I_1 + I_2)$.
\end{Proposition}

\begin{proof}
We prove this by showing that $U^{(1)} \cap U^{(2)} = \varnothing$. It suffices to show that $U^{(1)}_i \cap U^{(2)}_j = \varnothing$ for $i \le j$. Since $U^{(i)}_1 \subseteq E^\infty$, and $U^{(j)}_k \subseteq E^*$ for $k > 1$, $U^{(i)}_1 \cap U^{(j)}_k = \varnothing$ for all $i,j$ and all $k > 1$. By equation \eqref{main theorem remark one}, $U^{(1)}_1 \cap U^{(2)}_1 = \varnothing = U^{(1)}_2 \cap U^{(2)}_2$. Next
\[
U^{(1)}_2 \cap U^{(2)}_3
= E^*((H_{F_1,E} \setminus A) \cap (A_2 \setminus A))
= E^*((H_{F_1,E} \cap A_2) \setminus A)
= \varnothing,
\]
since $H_{F_1,E} \cap A_2 = A_2 \setminus F_0^0 \subseteq A$, by equation \eqref{pushout eqn two}. Finally,
\[
U^{(1)}_3 \cap U^{(2)}_3
= F_1^*(A_1 \setminus A) \cap F_2^*(A_2 \setminus A)
= E^*\bigl((A_1 \setminus A) \cap A_2 \setminus A)\bigr)
= E^*\bigl((A_1 \cap A_2) \setminus A\bigr)
= \varnothing,
\]
again by equation \eqref{pushout eqn two}.
\end{proof}

\begin{Proposition} \label{prop pullback 2}

$I_1 + I_2 = I_{12}$, and thus $\CT C^*(E,A) / (I_1 + I_2)$ is a relative Toeplitz graph algebra.

\end{Proposition}

\begin{proof}
We prove this by showing that $U^{(1)} \cup U^{(2)} = U^{(12)}$. Since $H_{F_1,E} \cup H_{F_2,E} = H_{F_0,E}$ it follows that $U^{(1)}_1 \cup U^{(2)}_1 = U^{(12)}_1$ and $U^{(1)}_2 \cup U^{(2)}_2 = U^{(12)}_2$. Next, from equation \eqref{pushout eqn two} we have that
\[
(A_1 \cup A_2) \setminus A = (A_1 \cup A_2) \cap F_0^0 \setminus A = A_{12} \setminus A.
\]
Then
$U^{(1)}_3 \cup U^{(2)}_3 = E^*\bigl((A_1 \cup A_2) \setminus A\bigr)= E^*(A_{12} \setminus A) = U^{(12)}_3$.
\end{proof}

\begin{Proposition} \label{prop pullback 3}

$I_{12} \subseteq I_0$.

\end{Proposition}

\begin{proof}
Since $A_{12} \subseteq A_0$, Notation \ref{main proofs notn}\eqref{u decomp notn} implies that $U^{(12)} \subseteq U^{(0)}$. This is equivalent to the containment of the proposition.
\end{proof}

\begin{Definition} \label{def admissible}

Let $\alpha_i : (F_0,A_0) \to (F_i,A_i)$ be morphisms of relative graphs, for $i = 1$, 2, and let $(E,A)$ be the pushout relative graph as in Definition \ref{definition relative graph}. The pair $(\alpha_1,\alpha_2)$ is called \textit{admissible} if 
\begin{equation} \label{eqn admissible}
A_0 \subseteq A_1 \cup A_2.
\end{equation}

\end{Definition}

\begin{Remark}

We wish to give some idea of what this condition means and why it is crucial for the pullback construction. Explicitly equation \eqref{eqn admissible} requires that if the Cuntz-Krieger condition is imposed at a vertex $v$ in $\CT C^*(F_0,A_0)$ then it is necessary to impose it at $v$ in at least one of $\CT C^*(F_i,A_i)$, $i = 1,2$ (and in particular, $v$ must be regular in at least one of $F_1$ and $F_2$). Let us also try to give a more fundamental explanation. For $i = 0,1,2$ let $I_i$ be the kernel of the quotient map of $\CT C^*(E,A)$ onto $\CT C^*(F_i,A_i)$. Let $v \in A_0$. Let us write $D_0 = v F_0^1$, $D_i = v (F_i^1 \setminus F_0^1)$ for $i = 1,2$, $q_0 = \sum_{e \in D_0} s_e s_e^*$, and (heuristically) $q_i = \sum_{e \in D_i} s_e s_e^*$. With the groupoid picture in mind we let $\chi_{\{v\}} = p_v - q_0 - q_1 - q_2$. Since $v \in A_0$ we have $p_v - q_0 = p_v - \sum_{e \in v F_0^1} s_e s_e^* = 0$ in $\CT C^*(F_0,A_0)$. Therefore $\chi_{\{v\}} + q_1 + q_2 = p_v - q_0 \in I_0$. We need that it also belong to $I_1 + I_2$ (as required by Theorem \ref{theorem pullback}).

First suppose that $v \in \sing{E}$. Then $v$ must be singular in at least one of $F_1$ and $F_2$. Suppose for definiteness that $v \in \sing{F_1}$. Then $v \not\in A_1$, since $A_1 \subseteq \reg{F_1}$. For $e \in D_2$ we know that $s(e) \in H_{F_1,E}$, and hence $s_e \in I_1$. Thus if $v \in \reg{F_2}$ then $q_2 \in I_1$. If in addition $v \in A_2$ then $q_1 = p_v - q_0 - q_2 \in I_2$, and it follows that $p_v - q_0 \in I_1 + I_2$. But if $v \not\in A_2$, or if $v \in \sing{F_2}$, i.e. if \eqref{eqn admissible} fails, then $\chi_{\{v\}} + q_1 + q_2$ cannot be apportioned between $I_1$ and $I_2$.

Next suppose that $v \in \reg{E}$. We use the same notations as above. Again, since $v \in A_0$ we have $\chi_{\{v\}} + q_1 + q_2 = p_v - q_0 \in I_0$. Since $v \in \reg{E}$ we know that $D_1$ and $D_2$ are finite, so $q_1$ and $q_2$ are not ``heuristic''. Therefore $q_1 \in I_2$ and $q_2 \in I_1$, by the same reasoning used for $q_2$ previously. If $v$ is in $A_1$ or in $A_2$, it follows as before that $p_v - q_0 \in I_1 + I_2$. But again, if $v \not\in A_1 \cup A_2$ then this is not possible.
\end{Remark}

\begin{Theorem} \label{thm main}
Let $\alpha_i : (F_0,A_0) \to (F_i,A_i)$ be morphisms of relative graphs, for $i = 1$, 2, and let $(E,A)$ be the pushout. We use the notation of \ref{main proofs notn}. The following statements are equivalent:

\begin{enumerate}[(a)]

\item \label{thm main 1} The commuting square of relative Toeplitz graph $C^*$-algebras corresponding to the pushout of $(\alpha_1,\alpha_2)$ is a pullback. (This refers to the outer square in Figure \ref{figure 3}.)
\item \label{thm main 2} $(\alpha_1,\alpha_2)$ is admissible.
\item \label{thm main 3} $A_0 = A_{12}$.
\item \label{thm main 4} $I_0 = I_{12}$.
\item \label{thm main 5} $A_0 \cap \sing{F_1} \subseteq A_2$, $A_0 \cap \sing{F_2} \subseteq A_1$, and $A_0 \cap \reg{E} \subseteq A_1 \cup A_2$.
\end{enumerate}

\end{Theorem}

\begin{proof}
\textit{\eqref{thm main 2} $\Leftrightarrow$ \eqref{thm main 3}:} Recall from Notation \ref{main proofs notn} that $A_{12} = (A_1 \cup A_2) \cap F_0^0 \subseteq A_0 \subseteq F_0^0$. If $(\alpha_1,\alpha_2)$ is admissible then $A_0 \subseteq A_1 \cup A_2$. Since $A_0 \subseteq F_0^0$ it follows that $A_0 \subseteq A_{12}$, hence $A_0 = A_{12}$. Conversely, if $A_0 = A_{12}$ then since $A_{12} \subseteq A_1 \cup A_2$ it follows that $(\alpha_1,\alpha_2)$ is admissible.

\textit{\eqref{thm main 1} $\Leftrightarrow$ \eqref{thm main 4}:} By Theorem \ref{theorem pullback} and Figure \ref{figure 3}, \eqref{thm main 1} is equivalent to the equality $I_1 + I_2 = I_0$. By Proposition \ref{prop pullback 2} this is equivalent to \eqref{thm main 4}.

\noindent
\textit{\eqref{thm main 3} $\Leftrightarrow$ \eqref{thm main 4}:} This follows from Remark \ref{rmk U's and I's}.

\noindent
\textit{\eqref{thm main 3} $\Leftrightarrow$ \eqref{thm main 5}:} Suppose that \eqref{thm main 3} holds. Let $v \in A_0$. Then $v \in A_i \subseteq \reg{F_i}$ for $i=1$ or $i=2$. If $v \in \sing{F_1}$ then $v \not\in A_1$, hence $v \in A_2$. Similarly, if $v \in \sing{F_2}$ then $v \in A_1$. The third condition is immediate. Next suppose that \eqref{thm main 5} holds. Let $v \in A_0$. Since $\sing{E} \subseteq \sing{F_1} \cup \sing{F_2}$, if $v \in \sing{E}$ then the first two conditions in \eqref{thm main 5} imply that $v \in A_2 \cup A_1$. If $v \in \reg{E}$ then the third condition implies that $v \in A_1 \cup A_2$.
\end{proof}

\section{Examples} \label{section examples}

We first consider the two extreme possibilities. Let $F_0 \hookrightarrow F_i$ for $i = 1,2$, and let $E$ be as in \eqref{pushout eqn one}.

\begin{Example}

Let $A_i = \varnothing$ for $i = 0,1,2$. Definition \ref{definition relative graph}\eqref{relative graph 3} reduces to $\varnothing \subseteq \varnothing$, so we do indeed have morphisms of relative graphs. It is clear that equation \eqref{pushout eqn two} holds, so we obtain a pushout diagram in the category of relative graphs. The corresponding commuting square of $C^*$-algebras consists of Toeplitz graph algebras (see Figure \ref{figure 4}).
\begin{figure}[h]
\begin{tikzpicture}[scale=2.5]

\node (0_0) at (0,0) [rectangle] {$\CT C^*(F_2)$};
\node (0_1) at (0,1) [rectangle] {$\CT C^*(E)$};
\node (1_0) at (1,0) [rectangle] {$\CT C^*(F_0)$.};
\node (1_1) at (1,1) [rectangle] {$\CT C^*(F_1)$};

\draw[-latex,thick] (0_1) -- (0_0);
\draw[-latex,thick] (0_1) -- (1_1);
\draw[-latex,thick] (1_1) -- (1_0);
\draw[-latex,thick] (0_0) -- (1_0);

\end{tikzpicture}
\captionof{figure}{}
\label{figure 4}
\end{figure}
The condition in Definition \ref{def admissible} reduces to $\varnothing \subseteq \varnothing$, so we know that this is a pullback diagram of $C^*$-algebras. (This is a special case of \cite[Theorem 3.3]{kpsw}.)

\end{Example}

\begin{Example} \label{example CK algebras}

Let $A_i = \reg{F_i}$ for $i = 0,1,2$. As noted in Corollary \ref{cor kernel of quotient map}, if the inclusion $(F_0,\reg{F_0}) \hookrightarrow (F_i,\reg{F_i})$ is a morphism then $H_{F_0,F_i}$ must be saturated in $F_i$. It then follows that $H_{F_i,E}$ is saturated in $E$ for $i = 1,2$. We must check Definition \ref{definition relative graph}\eqref{relative graph 3}. First, if $v \in \reg{E} \cap F_1^0$, the fact that $H_{F_1,E}$ is saturated in $E$ implies that $v$ is not a source in $F_1^0$. Therefore $v \in \reg{F_1}$. Similarly we have $\reg{E} \cap F_2^0 \subseteq \reg{F_2}$. Next, if $v \in \reg{F_1} \cap F_0^0$, then the fact that $H_{F_0,F_1} = H_{F_2,E}$ is saturated in $F_1$ implies that $v$ is not a source in $F_0$. Therefore $v \in \reg{F_0}$, and similarly, $\reg{F_2} \cap F_0^0 \subseteq \reg{F_0}$. Thus $(F_i,\reg{F_i}) \hookrightarrow (E,\reg{E})$ are, in fact, morphisms of relative graphs. Next we check equation \eqref{pushout eqn two} to verify that we have a pushout diagram in the category of relative graphs. We must show that $\reg{E} = (\reg{F_1} \setminus F_0^0) \cup (\reg{F_2} \setminus F_0^0) \cup (\reg{F_1} \cap \reg{F_2})$. For $\supseteq$, note that $\reg{F_1} \setminus F_0^0 \subseteq F_1^0 \setminus F_0^0 = H_{F_0,F_1} = H_{F_2,E}$. So if $v \in \reg{F_1} \setminus F_0^0$, then by the hereditary property of $H_{F_2,E}$, $v F_2^1 = \varnothing$, i.e. $v F_1^1 = v E^1$, and hence $v \in \reg{E}$. Similarly we have that $\reg{F_2} \setminus F_0^0 \subseteq \reg{E}$. Finally, it is clear that $\reg{F_1} \cap \reg{F_2} \subseteq \reg{E}$. For $\subseteq$, let $v \in \reg{E}$. If $v \in H_{F_1,E}$ then $v E^1 = v F_1^1$, so that $v \in \reg{F_1} \setminus F_0^0$. Similarly, if $v \in H_{F_2,E}$ then $v \in \reg{F_2} \setminus F_0^0$. Finally, let $v \in F_0^0$. If, say, $v \not\in \reg{F_1}$, then it must be a source in $F_1$. This means that $v E^1 = v F_2^1 H_{F_1,E}$. By the saturation property, $v \in H_{F_1,E}$, a contradiction. Thus $v \in \reg{F_1}$, and a similar argument shows that $v \in \reg{F_2}$.

Now we consider admissibility. Theorem \ref{thm main}\eqref{thm main 5} becomes: $\reg{F_0} \cap \sing{F_1} \subseteq \reg{F_2}$, $\reg{F_0} \cap \sing{F_2} \subseteq \reg{F_1}$, and $\reg{F_0} \cap \reg{E} \subseteq \reg{F_1} \cup \reg{F_2}$. The third of these is automatically true, while the first two are equivalent to each other, and to the condition $\sing{F_1} \cap \sing{F_2} \cap \reg{F_0} = \varnothing$. In words, a vertex that is regular in $F_0$ cannot receive infinitely many edges from both of $H_{F_1,E}$ and $H_{F_2,E}$ (or equivalently, a vertex cannot be breaking for both of $H_{F_1,E}$ and $H_{F_2,E}$). Thus we find that the commuting square of graph algebras in Figure \ref{figure 5} is a pullback diagram if and only if $\sing{F_1} \cap \sing{F_2} \cap \reg{F_0} = \varnothing$.
\begin{figure}
\begin{tikzpicture}[scale=2.5]

\node (0_0) at (0,0) [rectangle] {$C^*(F_2)$};
\node (0_1) at (0,1) [rectangle] {$C^*(E)$};
\node (1_0) at (1,0) [rectangle] {$C^*(F_0)$.};
\node (1_1) at (1,1) [rectangle] {$C^*(F_1)$};

\draw[-latex,thick] (0_1) -- (0_0);
\draw[-latex,thick] (0_1) -- (1_1);
\draw[-latex,thick] (1_1) -- (1_0);
\draw[-latex,thick] (0_0) -- (1_0);

\end{tikzpicture}
\captionof{figure}{}
\label{figure 5}
\end{figure}
We remark that if we let $A_0 = \reg{F_0} \setminus (\sing{F_1} \cap \sing{F_2})$, then the diagram becomes that of Figure \ref{figure 6}, which is a pullback.
\begin{figure}
\begin{tikzpicture}[scale=2.5]
\node (0_0) at (0,0) [rectangle] {$C^*(F_2)$};
\node (0_1) at (0,1) [rectangle] {$C^*(E)$};
\node (1_0) at (1,0) [rectangle] {$\CT C^*(F_0,A_0)$.};
\node (1_1) at (1,1) [rectangle] {$C^*(F_1)$};
\draw[-latex,thick] (0_1) -- (0_0);
\draw[-latex,thick] (0_1) -- (1_1);
\draw[-latex,thick] (1_1) -- (1_0);
\draw[-latex,thick] (0_0) -- (1_0);
\end{tikzpicture}
\captionof{figure}{}
\label{figure 6}
\end{figure}

\end{Example}

\begin{Example} \label{example compare}

We now compare our notion of admissible with that in \cite{hrt}. We describe the differences in our approaches. The situation in \cite{hrt} is of a graph $E$ and two subgraphs $F_1$ and $F_2$. In order to have a pushout of graphs it is necessary to assume that $E = F_1 \cup F_2$. Then $E$ is the pushout of the diagram $F_1 \cap F_2 \hookrightarrow F_i$, $i = 1,2$. In order that the incusions $F_i \hookrightarrow E$ define quotient maps $C^*(E) \to C^*(F_i)$ it is necessary that $E^0 \setminus F_i^0$ be saturated and hereditary in $E$ for $i = 1,2$. In order that the inclusions $F_1 \cap F_2 \hookrightarrow F_i$ define quotient maps $C^*(F_i) \to C^*(F_1 \cap F_2)$ it is necessary that $F_i^0 \setminus (F_1^0 \cap F_2^0) = E^0 \setminus F_j^0$ be saturated and hereditary in $F_i$, for $i = 1,2$ and $i \not= j$. In fact, the quotient of $C^*(F_i)$ by the ideal corresponding to the saturated hereditary set $F_i^0 \setminus (F_1^0 \cap F_2^0)$ equals the algebra of the graph $F_0$ defined by $F_0^0 = F_1^0 \cap F_2^0$ and $F_0^1 = F_0^0 E^1 F_0^0$. Thus it is necessary that $F_1 \cap F_2 = F_0$. In order to impose these requirements, the pair $F_1$ and $F_2$ is called \textit{admissible} (\cite[Definition 2.1]{hrt}) if, adjusting for the Australian convention, and using the graph $F_0$ defined above,

\begin{enumerate}[(1)]

\item $E = F_1 \cup F_2$

\item $\source{F_1 \cap F_2} \subseteq \source{F_1} \cap \source{F_2}$

\item $F_1^1 \cap F_2^2 = F_i^1F_0^0$, (equivalently, $F_1^1 \cap F_2^1 = E^1 F_0^0$, as is easily checked)

\item $F_i^0 \setminus F_0^0$ has no breaking vertices in $E$ or in $F_i$, $i = 1,2$.

\end{enumerate}

Given (1), \cite[Lemma 2.2]{hrt} implies that (3) is equivalent to the hereditary property of $F_i^0 \setminus F_0^0$ in $F_i$ (and hence in $E$) together with the condition $F_1 \cap F_2 = F_0$ mentioned above. Furthermore, \cite[Lemma 2.3]{hrt} shows that (2) implies $F_i^0 \setminus F_0^0$ is saturated in $F_i$, $i = 1,2$. Given (1) it also follows that $F_i^0 \setminus F_0^0$ is saturated in $E$. We remark that (2) is stronger than needed for saturation. It may be replaced by the weaker condition: (2)$'$ $\source{F_1 \cap F_2} \subseteq \sing{F_1} \cap \sing{F_2}$. In fact, (2)$'$ is equivalent to the saturation of $F_i^0 \setminus F_0^0$ in $F_i$, $i = 1,2$.

In the present paper we begin with two graphs, $F_1$ and $F_2$, and a common subgraph $F_0$. We \textit{construct} the graph $E$ as the pushout of the diagram $F_0 \hookrightarrow F_i$, $i = 1,2$. Then the condition $F_0 = F_1 \cap F_2$ holds by definition. We include the requirement that $F_i^0 \setminus F_0^0$ be hereditary in $F_i$ as part of our definition of morphism of (relative) directed graphs. Thus (1),(2)$'$,(3) are equivalent to our setup (in the special case that $A_i = \reg{F_i}$ for $i = 0,1,2$). Since we allow the graphs to have breaking vertices, we do not include a version of (4). Instead, our more general context requires our version of admissibility, namely, $\reg{F_0} \subseteq \reg{F_1} \cup \reg{F_2}$. Equivalently, a vertex in $F_0$ cannot be a breaking vertex for $F_i^0 \setminus F_0^0$ in $F_i$ for both $i = 1,2$ (as in Example \ref{example CK algebras}). We give the following examples.

\end{Example}

\begin{Example}

\begin{figure}
\begin{tikzpicture}

\node (0_0) at (0,0) [rectangle] {$u$};
\node (1_0) at (1.7,0) [rectangle] {$w,$};
\node(m1_0) at (-1.7,0) [rectangle] {$v$};

\draw[thick,->] (-1.5,0) to node[above]{$e_i$} (-.2,0);
\draw[thick,->] (1.5,0) to node[above]{$f_i$} (.2,0);

\draw[->,thick] (0,-.75) .. controls (.5,-.75) and (.5,-.3) .. (0_0) node[pos=0, inner sep=0.5pt, anchor=north] {$d$};

\draw[thick] (0,-.75) .. controls (-.5,-.75) and (-.5,-.3) .. (0_0);

\node (more) at (3.5,0) [rectangle] {$i = 1,2,\ldots$};

\end{tikzpicture}
\captionof{figure}{}
\label{figure 7}
\end{figure}

Let $E$ be the graph in Figure \ref{figure 7}. Let $F_0,F_1,F_2$ be the graphs defined by
\begin{align*}
&F_0^0 = \{u\},\ F_0^1 = \{d\} \\
&F_1^0 = \{u,v\},\ F_1^1 = \{d\} \cup \{e_i : i \in \IN\} \\
&F_2^0 = \{u,w\},\ F_2^1 = \{d\} \cup \{f_i : i \in \IN\}.
\end{align*}
Here $H_{F_0,F_1} = \{v\}$, $H_{F_0,F_2} = \{w\}$, and $u$ is a breaking vertex in $F_i$ for $H_{F_0,F_i}$, $i = 1,2$. Since there is a breaking vertex, this is not an admissible pushout according to \cite{hrt}. In this paper, the pushout diagram is not admissible because $u \in \reg{F_0} = A_0$ but $u \not\in \reg{F_i} = A_i$ for $i = 1,2$. Theorem \ref{thm main} implies that the corresponding diagram of graph $C^*$-algebras is not a pullback.

\end{Example}

\begin{Example}

\begin{figure}
\begin{tikzpicture}

\node (0_0) at (0,0) [rectangle] {$u$};
\node (1_0) at (1.7,0) [rectangle] {$w$,};
\node(m1_0) at (-1.7,0) [rectangle] {$v$};

\draw[thick,->] (-1.5,0) to node[above]{$e$} (-.2,0);
\draw[thick,->] (1.5,0) to node[above]{$f_i$} (.2,0);

\draw[->,thick] (0,-.75) .. controls (.5,-.75) and (.5,-.3) .. (0_0) node[pos=0, inner sep=0.5pt, anchor=north] {$d$};

\draw[thick] (0,-.75) .. controls (-.5,-.75) and (-.5,-.3) .. (0_0);

\node (more) at (3.5,0) [rectangle] {$i = 1,2,\ldots$};

\end{tikzpicture}
\captionof{figure}{}
\label{figure 8}
\end{figure}

Let $E$ be the graph in Figure \ref{figure 8}. Let $F_0,F_1,F_2$ be the graphs defined by
\begin{align*}
&F_0^0 = \{u\},\ F_0^1 = \{d\} \\
&F_1^0 = \{u,v\},\ F_1^1 = \{d,e\} \\
&F_2^0 = \{u,w\},\ F_2^1 = \{d\} \cup \{f_i : i \in \IN\}.
\end{align*}
Again, $H_{F_0,F_1} = \{v\}$, $H_{F_0,F_2} = \{w\}$. This time $u$ is a breaking vertex for $H_{F_0,F_2}$ in $F_2$, but not for $H_{F_0,F_1}$ in $F_1$. Since there is a breaking vertex, this is not an admissible pushout according to \cite{hrt}. In this paper, the pushout diagram is admissible because $u \in \reg{F_0} \cap \reg{F_1}$. Therefore Theorem \ref{thm main} implies that the corresponding diagram of graph algebras is a pullback.

\end{Example}

\end{document}